\documentclass[12pt, a4paper]{amsart}
\usepackage{hyperref, fullpage, amsmath, amsfonts, amsthm, amssymb, stmaryrd}

\theoremstyle{plain}
\newtheorem{Theorem}{Theorem}[section]
\newtheorem{Lemma}[Theorem]{Lemma}
\newtheorem{Fact}[Theorem]{Fact}
\newtheorem{Proposition}[Theorem]{Proposition}

\theoremstyle{definition}
\newtheorem{Definition}[Theorem]{Definition}

\theoremstyle{remark}
\newtheorem{Remark}[Theorem]{Remark}

\author{Arno Fehm}
\address{Fachbereich Mathematik und Statistik, University of 
Konstanz, 78457 Konstanz, Germany}
\email{arno.fehm@uni-konstanz.de}

\author{Franziska Jahnke}
\address{Institut f\"ur Mathematische Logik und Grundlagenforschung,
University of M\"unster, Einsteinstr.\;62, 48149 M\"unster, Germany}
\email{franziska.jahnke@wwu.de} 

\newcommand{\rk}{{\rm rk}}

\begin{document}

\title{Fields with almost small absolute Galois group}

\begin{abstract}
We construct and study fields $F$ with the property that $F$ has infinitely many extensions of some fixed degree, 
but $E^\times/(E^\times)^n$ is finite for every finite extension $E/F$ and every $n\in\mathbb{N}$.
\end{abstract}

\maketitle

\section{Introduction}

\noindent
We study the following closely related algebraic conditions on a field $F$:

\begin{enumerate}
\item[$\mathrm{(F1)}$] For every $n\in\mathbb{N}$, the field $F$ has only finitely many extensions of degree $n$
(sometimes referred to as {\em $F$ is bounded}).
\item[$\mathrm{(F2)}$] For every $n\in\mathbb{N}$ and every finite extension $E$ of $F$, the subgroup of $n$-th powers $(E^\times)^n$ 
has finite index in the multiplicative group $E^\times$.
\end{enumerate}
Both conditions appear already in Serre's {\em Cohomologie Galoisienne} \cite[Ch.~III \S4]{Serre}
and recently acquired importance in the model theory of fields:
For example, it is known that 
every supersimple field satisfies $\mathrm{(F1)}$ (Pillay-Poizat),
and for perfect pseudo-algebraically closed fields also the converse holds (Hrushovski).
Condition $\mathrm{(F2)}$ is satisfied by every superrosy field and also by every strongly$^2$-dependent field,
and it appears in a conjecture of Shelah-Hasson on definable valuations in NIP fields,
as well as in related results by Krupi\'nski.
For details on all of this see \cite{Krupinski}, \cite[Cor.~2.7]{KaplanShelah}.

It is well-known (we recall this in Proposition \ref{prop:F1F2} below) that
 $\mathrm{(F1)}$ 
implies  $\mathrm{(F2)}$, it was however previously not known 
whether the converse holds.
For example, finite or pseudo-finite fields and local fields like $\mathbb{R}$ and $\mathbb{Q}_p$ are known to satisfy both $\mathrm{(F1)}$ and $\mathrm{(F2)}$,
while global fields like $\mathbb{Q}$ or $\mathbb{F}_q(t)$ satisfy neither of them.
Similarly, while $\mathrm{(F1)}$ is clearly preserved under elementary equivalence of fields,
it was not known whether $\mathrm{(F2)}$ is.

We answer both questions negatively:

\begin{Theorem}\label{thm:intro1}
If a field $F$ satisfies $\mathrm{(F2)}$ and $F^*$ is a field elementarily equivalent to $F$,
then $F^*$ need not satisfy $\mathrm{(F2)}$.
\end{Theorem}

\begin{Theorem}\label{thm:intro2}
Even if all fields elementarily equivalent to $F$ satisfy
 $\mathrm{(F2)}$, $F$ need not satisfy $\mathrm{(F1)}$.
\end{Theorem}

\noindent
The theorems are proven by constructing counterexamples.
These counterexamples are obtained by first translating the problem into group theory and then realizing
suitable profinite groups -- the universal Frattini cover of products over certain finite groups derived from wreath products -- as absolute Galois groups.
The fields obtained by such a construction can be chosen either pseudo-algebraically closed or henselian valued.
In the last section we have a closer look at the henselian case and relate
 $\mathrm{(F1)}$ and $\mathrm{(F2)}$ to the residue field.

\section{Translation to group theory}
\label{sec:Kummer}

\noindent
We now explain the translation of $\mathrm{(F1)}$ and $\mathrm{(F2)}$ 
into properties of the absolute Galois group $G_F$ of $F$ and recall why 
 $\mathrm{(F1)}$ implies $\mathrm{(F2)}$.
For simplicity, we will from now on always assume that 
{\em $F$ is of characteristic zero}.
Let $G$ be a profinite group and consider the following two conditions on $G$:
\begin{enumerate}\label{ax:G1}
\item[$\mathrm{(G1)}$] $G$ is a small profinite group, i.e.~for every $n\in\mathbb{N}$ there are only finitely many open subgroups $H\leq G$ 
 of index $n$.
\item[$\mathrm{(G2)}$] For every $n\in\mathbb{N}$, every open subgroup $H\leq G$ has only finitely many open normal subgroups $N\lhd H$
with $H/N$ cyclic of order $n$.
\end{enumerate}
For $\mathrm{(F1)}$ the translation follows directly from Galois correspondence:
\begin{Fact}\label{fact:F1G1}
$F$ satisfies $\mathrm{(F1)}$ if and only if $G=G_F$ satisfies 
$\mathrm{(G1)}$.
\end{Fact}

Let $E$ be a field of characteristic zero
and let $\overline{E}$ be an algebraic closure of $E$. 
For $n\in\mathbb{N}$ 
we denote by $\mu_n\subseteq\overline{E}$ the group of $n$-th roots of unity,
and let $\mu_\infty=\bigcup_{n\in\mathbb{N}}\mu_n$.

\begin{Lemma}\label{lem:F1F2}
If $G=G_E$ is small, then $E^\times/(E^\times)^n$ is finite for any $n\in\mathbb{N}$.
\end{Lemma}

\begin{proof}
The short exact sequence
$$
 1\rightarrow\mu_n\rightarrow\overline{E}^\times\stackrel{\cdot^n}{\rightarrow}\overline{E}^\times\rightarrow 1
$$
gives rise to the long cohomology sequence
$$
 1\rightarrow\mu_n^G\rightarrow E^\times\stackrel{\cdot^n}{\rightarrow}E^\times\rightarrow H^1(G,\mu_n)\rightarrow H^1(G,\overline{E}^\times)\rightarrow\dots
$$
Since $H^1(G,\overline{E}^\times)=1$ by Hilbert's Theorem 90 \cite[Ch.~VI Thm.~10.1]{Lang}, we conclude that $E^\times/(E^\times)^n\cong H^1(G,\mu_n)$.
Let $N=G_{E(\mu_n)}$, which is an open normal subgroup of $G$. 
The inflation-restriction sequence \cite[Ch.~XX Ex.~6]{Lang} 
$$
 1\rightarrow H^1(G/N,\mu_n^N)\stackrel{{\rm inf}}{\rightarrow} H^1(G,\mu_n)\stackrel{{\rm res}}{\rightarrow} H^1(N,\mu_n)
$$
shows that $H^1(G,\mu_n)$ is finite, as $H^1(G/N,\mu_n^H)$ is finite (since $G/N$ and $\mu_n$ are finite)
and $H^1(N,\mu_n)={\rm Hom}(N,\mu_n)$ is finite (since $N$ is small and $\mu_n$ is finite).
\end{proof}

\begin{Proposition}\label{prop:F1F2}
If $F$ satisfies $\mathrm{(F1)}$, then it satisfies $\mathrm{(F2)}$.
\end{Proposition}

\begin{proof}
Since $\mathrm{(F1)}$ 
implies that $G_E$ is small for every finite extension $E/F$, 
the claim follows from Lemma \ref{lem:F1F2}.
\end{proof}

In the case where $F$ contains all roots of unity, 
this follows more directly from the following considerations.

\begin{Lemma}\label{lem:F2G2}
Suppose that $\mu_n\subseteq E$.
Then $E^\times/(E^\times)^n$ is finite if and only if $E$ has only finitely many cyclic extensions of degree dividing $n$.
\end{Lemma}

\begin{proof}
Let $B\leq E^\times$ be a subgroup containing $(E^\times)^n$, 
and denote by $E_B$ the field obtained from $E$ by adjoining $n$-th roots of all elements of $B$.
By Kummer theory, the map $B\mapsto E_B$ gives a bijection between the set of such subgroups $B$ 
and the abelian extensions of $E$ of exponent $n$, and 
if $(B:(E^\times)^n)<\infty$, then
${\rm Gal}(E_B/E)\cong B/(E^\times)^n$, cf.~\cite[Ch.~VI \S8]{Lang}.
In particular, the cyclic subgroups of $E^\times/(E^\times)^n$ correspond to cyclic extensions of $E$ of degree dividing $n$.
Since $E^\times/(E^\times)^n$ has infinitely many cyclic subgroups if and only if it is infinite,
the claim follows.
\end{proof}

\begin{Proposition}\label{prop:F2G2}
If $\mu_\infty\subseteq F$,
then $F$ satisfies $\mathrm{(F2)}$ if and only if $G=G_F$ satisfies 
$\mathrm{(G2)}$.
\end{Proposition}

\begin{proof}
This follows from Lemma \ref{lem:F2G2} applied to the finite extensions $E$ of $F$.
\end{proof}

In order to deal with the fields elementarily equivalent to $F$ we also 
need a uniform variant of $\mathrm{(G2)}$.
We denote by $C_n$ the cyclic group of order $n$.
We write $H\leq G$ and $H\lhd G$ to denote that $H$ is a closed respectively closed normal subgroup of $G$.

\begin{Definition}
For $n,m\in\mathbb{N}$ we let
$$
  I_G(n) = |\{N\lhd G : G/N\cong C_n \}|
$$
and
$$
 I_G(n,m) = \sup \{ I_H(n) :  H\leq G,(G:H)\leq m \}.
$$
\end{Definition}
With this definition, $\mathrm{(G2)}$ 
means that $I_H(n)<\infty$ for all open $H\lhd G$, 
and the uniform variant of $\mathrm{(G2)}$ can now be formulated as follows:
\begin{enumerate}
\item[$\mathrm{(G2^*)}$] For every $n,m\in\mathbb{N}$, $I_G(n,m)<\infty$.
\end{enumerate}
In other words, $G_F$ satisfies $\mathrm{(G2^*)}$ if and only if there is a uniform bound on the number of cyclic extensions 
of degree $n$ of finite extensions $E$ of $F$ of degree at most $m$.

\begin{Proposition}\label{prop:F2*G2*}
If $\mu_\infty\subseteq F$, 
then $G=G_F$ satisfies $\mathrm{(G2^*)}$ if and only if all fields 
$F^*\equiv F$ satisfy $\mathrm{(F2)}$.
\end{Proposition}

\begin{proof}
For every $m,n,k\in\mathbb{N}$,
there is a sentence $\varphi_{m,n,k}$ in the language of fields such that
$F\models \varphi_{m,n,k}$ if and only if every extension $E$ of $F$ with $[E:F]\leq m$
has at most $k$ cyclic extensions of degree $n$.

If $G_F$ satisfies $\mathrm{(G2^*)}$, 
then $F\models\varphi_{m,n,I_{G_F}(n,m)}$ for every $m,n$, 
so if $F^*\equiv F$, then also $F^*\models\varphi_{m,n,I_{G_F}(n,m)}$,
and therefore $I_{G_{F^*}}(n,m)\leq I_{G_F}(n,m)<\infty$ for all $m,n$.
In particular, $G_{F^*}$ satisfies $\mathrm{(G2)}$, so $F^*$ satisfies 
$\mathrm{(F2)}$ by Proposition \ref{prop:F2G2}.

Conversely, if $G_F$ does not satisfy $\mathrm{(G2^*)}$, 
then there exist $m,n$ such that
$F\models\neg\varphi_{m,n,k}$ for every $k$.
Let $F^*$ be an $\aleph_0$-saturated elementary extension of $F$. 
Then for every $k$, $F^*$ has an extension $E_k$ with $[E_k:F^*]\leq m$ which has more than $k$ cyclic extensions of degree $n$.
By saturation, $F^*$ has an extension $E^*$ with $[E^*:F^*]\leq m$ which has infinitely many cyclic extensions of degree $n$,
so $F^*$ does not satisfy $\mathrm{(F2)}$, by Lemma \ref{lem:F2G2}.
\end{proof}

\begin{Remark}
If $G$ is small, then the supremum in the definition of $I_G(n,m)$ runs over only finitely many $H$,
so $\mathrm{(G1)}$ implies $\mathrm{(G2^*)}$.
We thus have the following implications for a profinite group $G$:
$$
 G\mbox{ is finitely generated } \Longrightarrow\; \mathrm{(G1)} 
\;\Longrightarrow\; \mathrm{(G2^*)} \;\Longrightarrow\; \mathrm{(G2)}
$$
For the first implication see \cite[16.10.2]{FJ}.
It is well-known that the first implication cannot be reversed (see \cite[16.10.4]{FJ}),
and what we show in the next section is that the same holds for the other two implications.

We remark without proof that $G$ satisfies $\mathrm{(G2)}$ if and only if every open subgroup of $G$ has only finitely many {\em solvable} quotients of given order $n$, 
cf.~\cite[p.~III-30 Exercice]{Serre},
so if $G$ is pro-solvable, then the last two arrows are equivalences. 
Moreover, for pro-$p$ groups, all four conditions are equivalent, 
cf.~\cite[p.~III-28 Corollaire]{Serre}.
\end{Remark}

\section{Constructing profinite groups}

\noindent
We now construct a profinite group that satisfies 
$\mathrm{(G2^*)}$ but not $\mathrm{(G1)}$,
which is relatively straightforward,
and another one that satisfies $\mathrm{(G2)}$ but not $\mathrm{(G2^*)}$,
which requires more group theory.

\begin{Proposition}\label{lem:simple}
Let $S$ be any non-abelian finite simple group, $\kappa$ an infinite cardinal number, and $G=S^{\kappa}$.
Then $G$ satisfies $\mathrm{(G2^*)}$ but not $\mathrm{(G1)}$.
\end{Proposition}

\begin{proof}
Note that every open normal subgroup of $G$ is isomorphic to $G$ itself, with quotient of the form $S^k$ with $k\in\mathbb{Z}_{\geq0}$, cf.~\cite[Lemma 8.2.4]{RZ}.
In particular, $I_G(n)=0$ for all $n$.
If $H\leq G$ with $(G:H)\leq m$, let $N$ be the biggest normal subgroup of $G$ contained in $H$. 
Then $(G:N)\leq m!$ and $I_N(n)=0$ for all $n$. If $M\lhd H$ with $H/M\cong C_n$, then
$M\cap N\lhd N$ and $N/(M\cap N)\cong MN/M\leq H/M\cong C_n$ is cyclic, hence trivial.
Thus, $N\leq M\leq H$, and so $I_H(n)$ is bounded by the number of subgroups of $H/N$.
Therefore, $I_G(n,m)\leq 2^{m!}$.
Since $G$ has at least $\kappa$ many quotients isomorphic to $S$, it is not small.
\end{proof}

\begin{Lemma}\label{lem:reducetoprime}
Let $n\in\mathbb{N}$. Then $I_G(n)\leq 2^{n^s}$ with $s=\sum_{p|n\;\mathrm{prime}}I_G(p)$.
\end{Lemma}

\begin{proof}
Let $N_1,\dots,N_r\lhd G$ be distinct normal subgroups with $G/N_i\cong C_n$.
Let $N=\bigcap_{i=1}^r N_i$.
Then $A:=G/N$ embeds into $C_n^r$, hence 
$A\cong C_{d_1}\times\dots\times C_{d_k}$ with $k\in\mathbb{N}$ and $d_i|n$ for all $i$.
If $p|d_i$ is prime, then there is an epimorphism $\rho_i:C_{d_i}\rightarrow C_p$, and the maps
$$
 \delta_i:G\rightarrow A\stackrel{\cong}{\rightarrow}C_{d_1}\times\dots\times C_{d_k}\stackrel{\pi_i}{\rightarrow} C_{d_i}\stackrel{\rho_i}{\rightarrow} C_p
$$ 
are surjective and mutually distinct ($1\leq i\leq k$).
Thus, if $I_G(p)<\infty$ for all $p|n$, then $k\leq s:=\sum_{p|n}I_G(p)$.
As $N_1/N,\ldots,N_r/N$ are distinct subsets of $A$,
we see that
$r$ is bounded by the number of subsets of $A$.
Hence, $r\leq 2^{|A|}\leq 2^{n^s}$.
\end{proof}

\begin{Remark}\label{rem:rp}
Let $p$ be a prime number, and let $M_p(G)$ be the intersection over all $N\lhd G$ with $G/N\cong C_p$.
Then $G/M_p(G)\cong C_p^{r_p(G)}$, where $r_p(G)$ is the {\em $p$-rank} of $G$, cf.~\cite[Sec.~8.2]{RZ}.
Since $V=C_p^{r_p(G)}$ is an $\mathbb{F}_p$-vector space of dimension $r_p(G)$, and the $C_p$-quotients of $V$ correspond to
1-dimensional subspaces of the dual space $V^*$, 
we see that 
$$
 I_G(p)=|\mathbb{P}V^*|=\frac{p^{r_p(G)}-1}{p-1}
$$ 
if $r_p(G)$ is finite, and $I_G(p)=\infty$ otherwise.
We also see that if $F$ is a field of characteristic zero with $\mu_p\subseteq F$, then
$|F^\times/(F^\times)^p|=|G_F/M_p(G_F)|=p^{r_p(G_F)}$, cf.\ the proof of Lemma~\ref{lem:F2G2}.
\end{Remark}

\begin{Lemma}\label{lem:perfectbasis}
If a profinite group $G$ has, for every prime $p$, a basis of neighborhoods of $1$ consisting of open normal subgroups $U$
with $r_p(U)<\infty$,
then it satisfies $\mathrm{(G2)}$.
\end{Lemma}

\begin{proof}
Let $H\leq G$ be an open subgroup and let $p$ be a prime number.
By assumption, $H$ contains an open normal subgroup $U$ of $G$ with $r_p(U)<\infty$.
Thus,
$$
 U/(U\cap M_p(H))\cong M_p(H)U/M_p(H)\leq H/M_p(H)\cong C_p^{r_p(H)},
$$ 
so $M_p(U)\leq U\cap M_p(H)$,
which implies that $(H:M_p(H))\leq (G:U)\cdot (U:M_p(U))<\infty$, and hence $r_p(H)$ is finite.
Since this holds for every $p$, Lemma \ref{lem:reducetoprime} shows that $G$ satisfies $\mathrm{(G2)}$.
\end{proof}

Recall that a profinite group $G$ is {\em perfect} if $G'=G$, where $G'$ denotes the {\em closed} subgroup of $G$ generated by the commutators.
Thus, $G$ is perfect if and only if $r_p(G)=0$ for all primes $p$.

\begin{Lemma}\label{lem:perfectproduct}
Every product $G=\prod_{i\in I} G_i$ of finite perfect groups $G_i$ satisfies 
$\mathrm{(G2)}$.
\end{Lemma}

\begin{proof}
Note that the open normal subgroups $G_J=\prod_{i\in I\setminus J}G_i$, $J\subseteq I$ finite, form a basis of neighborhoods of $1$ of $G$.
Moreover, each $G_J$ is perfect as a product of perfect groups.
Thus, the claim follows from Lemma \ref{lem:perfectbasis}.
\end{proof}

\begin{Lemma}\label{lem:perfectextensions}
Let $S$ be a non-abelian finite simple group and $p$ a prime number.
For every $k_0$ there exists $k\geq k_0$ and a group extension of $S$ by $C_p^k$ which is perfect.
\end{Lemma}

\begin{proof}
Let $A=C_p^{k_0}$ and let $\Gamma=A\wr S$ be the wreath product, which is defined as the semidirect product $B\rtimes S$, 
where $S$ acts on the group $B$ of functions $f:S\rightarrow A$ by $f^\sigma(\tau)=f(\tau\sigma)$, where $\sigma,\tau\in S$.
Then $\Gamma''=\Gamma'=B_0\rtimes S$, where 
$$
 B_0=\left\{f\in B\;:\;\prod_{\sigma\in S}f(\sigma)=1\right\},
$$ 
see for example 
\cite[Ch.~I Cor.~4.9, Thm~4.11]{Meldrum}.
In particular, $\Gamma'$ is a perfect extension of $S$ by $B_0\cong C_p^k$, where $k:=k_0\cdot |S|-k_0\geq k_0$.
\end{proof}

\begin{Proposition}\label{prop:G2notG2*}
Let $S$ be a non-abelian finite simple group and $p$ a prime number.
Let $G$ be the product over all perfect extensions of $S$ by $C_p^k$ for all $k\in\mathbb{N}$.
Then $G$ satisfies $\mathrm{(G2)}$ but not $\mathrm{(G2^*)}$.
\end{Proposition}

\begin{proof}
By Lemma \ref{lem:perfectextensions},
for every $k_0$ there exists $k\geq k_0$ and a perfect extension $P$ of $S$ by $C_p^k$,
which by the definition is a quotient of $G$.
Since $P$ has an open subgroup $H$ of index $m=|S|$ with $r_p(H)\geq k$, so does $G$,
and therefore $I_G(p,m)\geq k\geq k_0$. Since this holds for every $k_0$, $G$ does not satisfy $\mathrm{(G2^*)}$.
Conversely, Lemma \ref{lem:perfectproduct} implies that $G$ satisfies 
$\mathrm{(G2)}$.
\end{proof}

\section{Constructing fields}

\noindent
We saw that the desired properties of fields are reflected by the properties 
$\mathrm{(G1)}$, $\mathrm{(G2)}$ and $\mathrm{(G2^*)}$ 
of their absolute Galois groups
and we already constructed suitable profinite groups.
However, the groups we constructed do not occur as absolute Galois groups of fields -- they have too much torsion.
Instead, we want to construct fields using the following result, cf.~\cite[23.1.2]{FJ}:

\begin{Proposition}[Lubotzky-van den Dries]\label{prop:LubotzkyDries}
For every field $K$ and every {\em projective} profinite group $G$ there is a perfect pseudo-algebraically closed field $F\supseteq K$
with $G_F\cong G$.
\end{Proposition}

In order to apply this result, we have to replace the profinite groups we constructed by projective ones with similar properties,
for which we will make use of the
{\em universal Frattini cover} $\tilde{G}$ of a profinite group $G$, cf.~\cite[Chapter 22]{FJ}.
We do not give the full definition but rather list the properties of $\tilde{G}$ that we need:
\begin{enumerate}
\item $\tilde{G}$ is a projective profinite group and 
there is an epimorphism
$\phi :\tilde{G} \rightarrow G$, see \cite[22.6.1]{FJ}.
\item For any quotient ${\Delta}$ of $\tilde{G}$ there is an 
epimorphism $\Delta \rightarrow \Gamma$ onto some
quotient $\Gamma$ of $G$ such that $\rk(\Delta)=\rk(\Gamma)$,
see~\cite[22.6.3, 22.5.3]{FJ}.
\end{enumerate}
Here, $\rk(G)$ denotes the profinite rank of $G$, cf.~\cite[Chapter 17.1]{FJ},
which for finite $G$ is just the minimal cardinality of a set of generators.
We now show that the properties $\mathrm{(G2)}$ and $\mathrm{(G2^*)}$ 
are preserved by taking the universal Frattini cover,
which is the technical heart of our construction.

\begin{Lemma}\label{lem:explicit}
For every prime $p$ and every $H\leq\tilde{G}$ with $(\tilde{G}:H)=m$ there exists $G_0\leq G$ with $(G:G_0)\leq m!$ such that
$r_p(H)\leq (m!)^2(r_p(G_0)+2)$.
\end{Lemma}

\begin{proof}
If $H_0$ is the biggest normal subgroup of $\tilde{G}$ contained in $H$, then 
$(\tilde{G}:H_0)\leq m!$. Furthermore, we have 
$$r_p(H)\leq r_p(H_0)+r_p(H/H_0)\leq r_p(H_0)+\log_p(m!)$$
with the first inequality following from \cite[8.2.5(d)]{RZ}.

Let $N=M_p(H_0)$. Since $H_0\lhd\tilde{G}$ and $N$ is characteristic in $H_0$, 
we conclude $N\lhd\tilde{G}$.
Let $\Delta=\tilde{G}/N$ and $\Delta_0=H_0/N\cong C_p^{r_p(H_0)}$.
By (2), there exist epimorphisms 
$\phi:\Delta\rightarrow\Gamma$, $\pi:G\rightarrow\Gamma$ with 
${\rm rk}(\Gamma)={\rm rk}(\Delta)$.
Let $\Gamma_0=\phi(\Delta_0)\lhd\Gamma$ and $G_0=\pi^{-1}(\Gamma_0)\lhd G$ 
and note that
$(G:G_0)=(\Gamma:\Gamma_0)$ divides $(\Delta:\Delta_0)=(\tilde G:H_0)\leq m!$.

Trivially, $r_p(G_0)\geq r_p(\Gamma_0)$.
Since $\Gamma_0$ is an elementary abelian $p$-group, we have
$\rk(\Gamma_0)=r_p(\Gamma_0)$ and so the inequality
$$
 \rk(\Gamma)\leq\rk(\Gamma_0)+\rk(\Gamma/\Gamma_0)\leq r_p(\Gamma_0)+m!
$$
holds.
By the Nielsen-Schreier formula \cite[17.6.3]{FJ}, we get
$$
 \rk(\Delta_0)\leq 1+(\Delta:\Delta_0)(\rk(\Delta)-1).
$$ 
Thus,
$$
 r_p(H_0)=\rk(\Delta_0)\leq 1+m!\cdot(\rk(\Gamma)-1)\leq 1+m!\cdot
(r_p(\Gamma_0)+m!-1)
\leq m! \cdot r_p(G_0)+(m!)^2,
$$
which 
gives 
$$
 r_p(H)\leq\log_p(m!)+m!\cdot r_p(G_0)+(m!)^2\leq (m!)^2(r_p(G_0)+2).
$$
\end{proof}

\begin{Proposition}\label{prop:Frattini}
The universal Frattini cover $\tilde{G}$ satisfies $\mathrm{(G2)}$ 
resp.~$\mathrm{(G2^*)}$ if and only if $G$ does.
\end{Proposition}

\begin{proof}
If $\tilde{G}$ satisfies $\mathrm{(G2)}$ or $\mathrm{(G2^*)}$, 
then so does its quotient $G$.

Conversely, assume that $G$ satisfies $\mathrm{(G2)}$ and let
$H\leq\tilde{G}$ with $(\tilde{G}:H)\leq m$.
By Lemma \ref{lem:explicit} there exists $G_0\leq G$ with $(G:G_0)\leq m!$ 
such that $r_p(H)$ is bounded in terms of $r_p(G_0)$ and $m$.
In particular, $I_H(p)$ is finite.
By Lemma \ref{lem:reducetoprime} we get for every $n$ that $I_H(n)$ is finite,
so $\tilde{G}$ satisfies $\mathrm{(G2)}$.

If $G$ satisfies even $\mathrm{(G2^*)}$ then 
$I_{G_0}(p)\leq I_G(p,m!)$ is uniformly bounded just in terms of $m$ and $p$, 
hence so is $r_p(G_0)$, and therefore also $I_H(p)$.
Thus, by Lemma \ref{lem:reducetoprime}, $I_H(n)$ is bounded in terms of $m$ and $n$,
so $I_{\tilde{G}}(n,m)<\infty$, which means that $\tilde{G}$ satisfies 
$\mathrm{(G2^*)}$.
\end{proof}

We now have all the ingredients to construct the counterexamples that prove
Theorem~\ref{thm:intro1} and Theorem~\ref{thm:intro2}:

\begin{Proposition}\label{prop:counterex1}
There exists a pseudo-algebraically closed field $F$ of characteristic zero such that
every $F^*\equiv F$ satisfies $\mathrm{(F2)}$, but $F$ does not satisfy 
$\mathrm{(F1)}$.
\end{Proposition}

\begin{proof}
Let $S$ be a non-abelian finite simple group, for example $S=A_5$, and let $G=S^{\aleph_0}$.
By Proposition \ref{lem:simple}, $G$ satisfies $\mathrm{(G2^*)}$ 
but not $\mathrm{(G1)}$.
Thus, by Proposition \ref{prop:Frattini}, also $\tilde{G}$ satisfies 
$\mathrm{(G2^*)}$, and, since it has $G$ as a quotient, it does not satisfy 
$\mathrm{(G1)}$.
Let $K$ be any field of characteristic zero that contains all roots of unity, for example $K=\mathbb{C}$.
By Proposition \ref{prop:LubotzkyDries} there exists a field $F\supseteq K$ which is pseudo-algebraically closed and has absolute Galois group $G_F\cong\tilde{G}$,
so all $F^*\equiv F$ satisfy $\mathrm{(F2)}$ by Proposition 
\ref{prop:F2*G2*}, but $F$ does not satisfy $\mathrm{(F1)}$ (Fact \ref{fact:F1G1}).
\end{proof}

\begin{Proposition}\label{prop:counterex2}
There exists a pseudo-algebraically closed field $F$ of characteristic zero that satisfies
$\mathrm{(F2)}$, but some $F^*\equiv F$ does not satisfy $\mathrm{(F2)}$.
\end{Proposition}

\begin{proof}
Let $S$ be a non-abelian finite simple group, for example $S=A_5$, let $p$ be any prime number, for example $p=2$,
and let $G$ be the product over all perfect extensions of $S$ by $C_p^k$ for all $k\in\mathbb{N}$.
By Proposition \ref{prop:G2notG2*}, $G$ satisfies $\mathrm{(G2)}$ 
but not $\mathrm{(G2^*)}$.
Thus, by Proposition \ref{prop:Frattini}, also $\tilde{G}$ satisfies 
$\mathrm{(G2)}$ but not $\mathrm{(G2^*)}$.
Let again $K$ be any field of characteristic zero that contains all roots of unity
and apply Proposition \ref{prop:LubotzkyDries} to get a field $F\supseteq K$ which is pseudo-algebraically closed and has absolute Galois group $G_F\cong\tilde{G}$.
By Proposition \ref{prop:F2G2}, $F$ satisfies $\mathrm{(F2)}$,
but by Proposition \ref{prop:F2*G2*} there is some $F^*\equiv F$ 
that does not satisfy $\mathrm{(F2)}$.
\end{proof}

\begin{Remark}
We remark that much more concrete realizations of projective profinite groups are known.
For example, since the groups we constructed have countable rank,
they could be realized as absolute Galois groups of {\em algebraic} extensions of $\mathbb{Q}$.
For instance, if $\mathbb{Q}^{\rm tr}$ denotes the field of totally real algebraic numbers -- the maximal Galois extension of $\mathbb{Q}$ in $\mathbb{R}$ --
then one can find algebraic extensions of $\mathbb{Q}^{\rm tr}(\mu_\infty)$ with the properties of Proposition \ref{prop:counterex1} or Proposition \ref{prop:counterex2},
cf.~\cite[Example 5.10.7]{Jarden}.
\end{Remark}

\section{Henselian fields}

\noindent
Since $\mathrm{(F1)}$ and $\mathrm{(F2)}$ 
are essentially properties of the absolute Galois group,
and every absolute Galois group occurs as the absolute Galois group of a henselian valued field,
it is clear that one can also construct such examples with $F$ henselian:

\begin{Proposition}
There exists a henselian valued field $F$ of characteristic zero that satisfies
$\mathrm{(F2)}$ but not $\mathrm{(F1)}$.
\end{Proposition}

\begin{proof}
Let $F$ be the field constructed in Proposition \ref{prop:counterex1},
and let $F'=F((\mathbb{Q}))$ be the field of generalized power series over $F$ with exponents in $\mathbb{Q}$, 
cf.~\cite[\S4.2]{Efrat}.
Then $F'$ is henselian valued with residue field $F$ and divisible value group $\mathbb{Q}$, see \cite[18.4.2]{Efrat}.
Thus, $G_{F'}\cong G_F$, as follows from \cite[5.2.7 and 5.3.3]{EP}.
Since $F'$ contains all roots of unity,
it satisfies $\mathrm{(F2)}$ but not $\mathrm{(F1)}$, as above.
\end{proof}

In this construction, the property that $F'$ satisfies $\mathrm{(F2)}$ but 
not $\mathrm{(F1)}$ is inherited from the residue field.
We now show that, at least in characteristic 0, 
this is in fact the only way to construct henselian fields with this property,
or, more generally, with properties like in Proposition \ref{prop:counterex1} and Proposition \ref{prop:counterex2}.
In order to do that, we need the following lemma.

\begin{Lemma}\label{lem:henselian}
Let $(F,v)$ be a henselian valued field 
with residue field $Fv$ of characteristic $0$ and value group $\Gamma$, and let $n\in\mathbb{N}$. Then
$$
 |F^\times/(F^\times)^n| = |\Gamma/n\Gamma| \cdot |Fv^\times/(Fv^\times)^n|
$$
holds. 
In particular, if $\mu_n\subseteq F$,
then $I_{G_F}(n)$ is finite if and only if both $[\Gamma:n\Gamma]$ and 
$I_{G_{Fv}}(n)$ are finite.
\end{Lemma}

\begin{proof}
Take $A=\{a_i\}_{i \in I} \subseteq 
\mathcal{O}_v$ such that $\{v(a_i)\}_{i\in I}$
form a system of representatives for $\Gamma/n\Gamma$ and
$B=\{b_i\}_{i \in J} \subseteq \mathcal{O}_v^\times$ such that 
$\{\overline{b_i}\}_{i\in J}$
form a system of representatives for $Fv^\times/(Fv^\times)^n$.

We first show 
$$|F^\times/(F^\times)^n| \geq 
|\Gamma/n\Gamma| \cdot |Fv^\times/(Fv^\times)^n|$$
for any valued field $(F,v)$:
Consider $(a,b)$, $(a',b') \in A \times B$. Assume that
we have $ab \equiv a'b'\textrm{ mod } (F^\times)^n$. Without loss of generality,
$ab= r^n a'b'$ for some $r \in \mathcal{O}_v$.
Then
$$
 v(a)= v(ab) = n v(r)+ v(a'b')=nv(r) + v(a'),
$$
so $a=a'$, as the values of $A$ form a system of representatives for 
$\Gamma/n\Gamma$. Thus,
$b=r^nb'$ holds, so we get 
$\overline{b} \equiv \overline{b'}\mbox{ mod }(Fv^\times)^n$ and hence $b=b'$,
which proves the claim.

On the other hand, take any $x \in F^\times$.
We want to show that there is some $(a,b) \in A 
\times B$ such that we have $xab \in (F^\times)^n$.
Choose $a \in A$ with $v(xa) \in n \Gamma$ and take some $u \in F^\times$
with $v(u^n)=v(xa)$. Then for $c = \frac{xa}{u^n}$ we get $v(c)=0$,
so there is some $b \in B$ with 
$\overline{t}^n = \overline{c}\overline{b}$ for some $t \in F^\times$.
By henselianity (see \cite[4.1.3]{EP}), 
$f(X)=X^n- \frac{cb}{t^n}$ has a zero $\alpha \in F^\times$,
as ${\rm char}(Fv)=0$. 
This
implies $xab = \alpha^n t^n u^n \in (F^\times)^n$. Thus
$|F^\times/(F^\times)^n| \leq |\Gamma/n\Gamma| \cdot |Fv^\times/(Fv^\times)^n|$
holds.

The last part now follows immediately from Lemma \ref{lem:F2G2}. 
\end{proof}

\begin{Proposition}
Let $(F,v)$ be a henselian valued field with residue field $Fv$ and value group $\Gamma$.
Assume that $\mathrm{char}(Fv)=0$ and $\mu_\infty\subseteq F$.
\begin{enumerate}
\item If $[\Gamma:p\Gamma]=\infty$ for some prime $p$, then $F$ 
satisfies neither $\mathrm{(F1)}$ nor $\mathrm{(F2)}$.
\item If $[\Gamma:p\Gamma]$ is finite for all primes $p$, then
\begin{enumerate}
\item $\mathrm{(F1)}$ holds for $F$ if and only if it holds for $Fv$,
\item $\mathrm{(F2)}$ holds for $F$ if and only if it holds for $Fv$, and
\item $\mathrm{(F2)}$ holds for every $K \equiv F$ if and only if 
it holds for every $k \equiv Fv$.
\end{enumerate}
\end{enumerate}
\end{Proposition}
\begin{proof}
(1)
Note that since $(F,v)$ is henselian of characteristic $(0,0)$ and 
$F$ contains all roots of unity, Lemma \ref{lem:henselian} applies.
Thus, $[\Gamma:p\Gamma]=\infty$ for some prime $p$ implies 
$|F^\times/(F^\times)^p|=\infty$ and so neither $\mathrm{(F1)}$ nor $\mathrm{(F2)}$ 
hold for $F$.

(2)
For the remainder of the proof, assume that 
$i_p :=[\Gamma:p\Gamma]$ is finite for all primes $p$.
By \cite[5.2.6 and 5.3.3]{EP}, we have
$$G_F \cong \left( \prod_{p \textrm{ prime}} \mathbb Z_p^{i_p} \right) 
\rtimes G_{Fv}.$$
\begin{enumerate}
\item[(a)]
Since $\prod_{p} \mathbb Z_p^{i_p}$ is small
and the class of small profinite groups is closed under semidirect products,
we get that $G_F$ is small if and only if $G_{Fv}$ is small,
i.e.~$F$ satisfies $\mathrm{(F1)}$ if and only if $Fv$ does (Fact \ref{fact:F1G1}).
\item[(b)]
If $G_F$ satisfies $\mathrm{(G2)}$
then so does its quotient $G_{Fv}$.

Conversely, assume that $\mathrm{(G2)}$ holds for $G_{Fv}$. 
Let $E$ be a finite extension of $F$, say $[E:F]=m$. 
Let $\Delta$ denote the value group and $Ev$ the residue field of the unique prolongation of $v$ to $E$. 
Define $f:=[Ev:Fv]$ and $e=[\Delta:\Gamma]$. Then -- by \cite[3.3.4]{EP} -- we have $ef\leq m$. 
For every prime $p$, $I_{G_{Ev}}(p)$ and thus $Ev^\times/(Ev^\times)^p$ is finite,
and $[\Delta:p\Delta] \leq [\Gamma:p\Gamma]\cdot e<\infty$,
so by applying Lemma \ref{lem:henselian} to $E$, we get 
$|E^\times/(E^\times)^p|<\infty$ for every $p$,
which by Remark \ref{rem:rp} and Lemma \ref{lem:reducetoprime} implies that $G_F$ satisfies $\mathrm{(G2)}$.
\item[(c)]
Again, if $G_F$ satisfies $\mathrm{(G2^*)}$, 
then so does its quotient $G_{Fv}$.
For the other direction, assume that $G_{Fv}$ satisfies $\mathrm{(G2^*)}$.
Fix any prime $p$ 
and let $E$ be a finite extension of $F$ with $[E:F]\leq m$ and define $Ev$,
$\Delta$ 
and $e=[\Delta:\Gamma]$ as before. Then,
making repeated use of Remark \ref{rem:rp}, we see that
\begin{align*}
I_{G_E}(p) &= \dfrac{p^{r_p(G_E)}-1}{p-1} 
	=\dfrac{1}{p-1}\cdot(|E^\times/(E^\times)^p|-1)\\
	&\stackrel{\scriptsize\ref{lem:henselian}}{\leq}\dfrac{1}{p-1}\cdot
([\Delta:p\Delta] \cdot |Ev^\times/(Ev^\times)^p|)
	\leq \dfrac{1}{p-1} 
([\Gamma:p\Gamma]\cdot e \cdot p^{r_p(G_{Ev})})\\
	&\leq
[\Gamma:p\Gamma]\cdot m \cdot (I_{G_{Ev}}(p)+1)
	\leq [\Gamma:p\Gamma]\cdot m \cdot (I_{G_{Fv}}(p,m)+1).
\end{align*}
Now Lemma \ref{lem:reducetoprime} implies that for any subgroup 
$H \leq G_F$ of index at most $m$, 
$I_H(n)$ is uniformly bounded in terms of $m$ and $n$,
i.e.~$I_{G_F}(n,m)<\infty$. 
Thus, $\mathrm{(G2^*)}$ 
holds also
for $G_F$, so any $K\equiv F$ satisfies $\mathrm{(F2)}$ (Proposition \ref{prop:F2*G2*}).
\end{enumerate}
\end{proof}

\section*{Acknowledgements}

\noindent
The authors would like to thank Lior Bary-Soroker for suggesting the use of the universal Frattini cover,
and Itay Kaplan, Krzysztof Krupi\'nski and Moshe Jarden for helpful discussions and comments on a previous version.

\end{document}